\documentclass{amsart}
\usepackage{amsmath,amsthm,amscd,amssymb, color, amsfonts, latexsym}
\usepackage{fullpage}
\usepackage{bm}
\usepackage{latexsym}
\usepackage{graphicx}
\usepackage[all]{xy}

\newtheorem{lem}{Lemma}[section]
\newtheorem{prop}[lem]{Proposition}
\newtheorem{thm}[lem]{Theorem}
\newtheorem{defn}[lem]{Definition}

\theoremstyle{definition}
\newtheorem{example}[lem]{Example}

\newcommand{\bmx}{\left( \begin{matrix}}
\newcommand{\emx}{\end{matrix} \right)}
\newcommand{\Q}{\mathbb Q}

\newcommand{\G}{\mathbb G}
\newcommand{\R}{ \mathbb R}
\newcommand{\C}{\mathbb C}
\newcommand{\Z}{\mathbb Z}

\newcommand{\Id}{\mathtt{Id}}

\date\today
\title{Explicit construction of Ramanujan bigraphs}
\author{Cristina Ballantine}
\thanks{This work was partially supported by a grant from the Simons Foundation (\#245997 to Cristina Ballantine)}
\thanks{This work was partially supported by grants   from the National Science Foundation (DMS-1201446)  and from PSC-CUNY (Brooke Feigon)}
\thanks{This work was partially supported by the European Social Fund (Kathrin Maurischat)}
\author{Brooke Feigon} \author{Radhika Ganapathy}\author{Janne Kool}\author{Kathrin Maurischat} \author{Amy Wooding}

\address{Department of Mathematics and Computer Science, College of
  the Holy Cross, 1 College Street, Worcester, MA 01610}
\email{cballant@holycross.edu}

\address{Department of Mathematics,
The City College of New York,
NAC 8/133,
New York, NY 10031}

\email{bfeigon@ccny.cuny.edu }

\address{Department of Mathematics,
The University of British Columbia
1984 Mathematics Road,
Vancouver, B.C.
Canada V6T 1Z2 }

\email{rganapat@math.ubc.ca}

\address{Max Planck Institute for Mathematics, Vivatsgasse 7,
53111 Bonn,
Germany}

\email{kool79@mpim-bonn.mpg.de}

\address{Mathematisches Institut,
Universit\"at Heidelberg,
Im Neuenheimer Feld 288,
69120 Heidelberg, Germany }

\email{maurischat@mathi.uni-heidelberg.de }

\address{The Department of Mathematics and Statistics,
McGill University,
Burnside Hall, Room 1023,
805 Sherbrooke W.,
Montreal, QC,
H3A 0B9, Canada}

\email{amy.cheung@mail.mcgill.ca }
\begin{document}

\begin{abstract} We construct explicitly an infinite family of Ramanujan graphs which are bipartite and biregular. 
Our construction starts with the Bruhat-Tits building of an inner form of $SU_3(\Q_p)$. 
To make the graphs finite, we take successive quotients by infinitely many discrete co-compact subgroups of decreasing size.  
\\
\\
{\bf MSC 2010:} 11E39, 11E57, 16W10

\end{abstract}
\maketitle

\section{Introduction}

Expander graphs are highly connected yet sparse graphs. By a highly connected graph we mean a graph in which all small sets of vertices have many neighbors. They have wide ranging applications, especially in computer science and 
coding theory. They also model neural connections in the brain and many other types of networks. One is usually interested in 
regular or biregular expanders. The expansion property is controlled by the size of the spectral gap of the graph. Asymptotically, 
Ramanujan graphs are optimal expanders as we will explain below.  Infinite families of regular Ramanujan graphs of fixed degree were first constructed in the late
1980's by Lubotzky, Phillips and Sarnak \cite{lps:1988}, and independently by Margulis \cite{margulis:1988}. 
Since then, the study of problems related to the existence and construction of Ramanujan graphs has become an active area 
of research. Until recently, all constructions of families of regular Ramanujan graphs have been obtained using tools from number 
theory, including deep results from the theory of automorphic forms.  As a result, the graphs obtained have degree $q+1$, 
where $q$ is a power of a prime. Using similar methods, the authors of \cite{ballantine:2011} give a roadmap toward the 
construction of infinite families of Ramanujan bigraphs, \textit{i.e.}, biregular, bipartite graphs satisfying the Ramanujan 
condition, of bidegree $(q^3+1, q+1)$, where $q$ is a power of a prime. However, they stop short of providing explicit examples. 
Very recently, Marcus, Spielman and Srivastava \cite{spielman:2014} used the method of interlacing polynomials to prove the 
existence of arbitrary degree Ramanujan bigraphs. By making the two degrees equal, this implies the existence of arbitrary degree (regular) Ramanujan graphs. 
Their proof is non-constructive.

In this article, we follow the roadmap given in \cite{ballantine:2011} to explicitly  construct an infinite family of 
Ramanujan bigraphs. We start with a quadratic extension, $E/\Q$, and define a division algebra $D$ which 
is non-split over $E$  \textit{i.e.}, $D$ is not isomorphic to the matrix algebra $M_3(E)$. We then use this to define a special unitary group $\G$
over $E$ from $D$ by means of an involution of the second kind. 
We define this involution such that  the corresponding local unitary
group is isomorphic to $SU_3(\mathbb{Q}_p)$ at the place $p$, 
\textit{i.e.}, $\G_p = \G(\Q_p) \cong SU_3(\mathbb{Q}_p)$,  and compact at infinity.  We also give a concrete description of an infinite family 
of discrete co-compact subgroups of $\G_p$ which act without fixed points on $\G_p$. 

Since the division algebra $D$ is non-split, Corollary 4.6 of \cite{ballantine:2011} guarantees that each quotient of the 
Bruhat-Tits tree of $\G_p$ by one of the above subgroups satisfies the Ramanujan condition.  Therefore, we obtain an infinite family
of Ramanujan bigraphs of bidegree $(p^3+1, p+1)$. We note that most of this work could be carried out over a general totally real 
number field but we often choose to work over $\mathbb Q$ to simplify the notation.

\medskip

\section*{Acknowledgements} The authors would like to thank Dan Ciubotaru, Judith Ludwig and Rachel Newton for helpful discussions. This work was initiated at the CIRM workshop, ``Femmes en nombre'' in October 2013 and the authors would like to thank CIRM, Microsoft Research, the National Science Foundation DMS-1303457, the Clay Mathematics Institute and the Number Theory Foundation for supporting the workshop.

\medskip

\section{Preliminaries and Notation}

In this section we introduce the notation used throughout the article and give a brief review of Ramanujan graphs and bigraphs, unitary groups, and buildings. 

\subsection{Ramanujan graphs and bigraphs} While \cite{ballantine:2011} also contains a concise review of this topic, we find it useful for the reader to have an overview within the current article. In \cite{lub:2012}, Lubotzky gives a  review of expander graphs with applications within mathematics.  Hoory, Linial and Wigderson \cite{hlw:2006} provide a review accessible to the nonspecialist with many applications, especially to computer science. For an elementary introduction to regular Ramanujan graphs, we refer the reader to \cite{davidoff:2003}.

A graph $X=(V,E)$ consists of a set of vertices $V$ together with a subset of pairs  of vertices called edges. In this article, all graphs are undirected. Thus, the pair of vertices forming an edge is unordered. The \textit{degree} of a vertex is the number of edges incident to it. A graph is called $k$-\textit{regular} if all vertices have degree $k$. A graph is called $(l,m)$-\textit{biregular} if each vertex has degree $l$ or $m$. A \textit{bipartite} graph is a graph that admits a coloring of the vertices with two colors such that no two adjacent vertices have the same color. A \textit{bigraph} is a biregular, bipartite graph.

We denote by Ad$(X)$ the adjacency matrix of $X$ and by Spec$(X)$ the spectrum of $X$. Thus, Spec$(X)$ is the collection of eigenvalues of Ad$(X)$. Since the adjacency matrix is symmetric, Spec$(X)\subset \mathbb R$. For a $k$-regular graph, we have $k \in$ Spec$(X)$. For an $(l,m)$-biregular graph, we have $\sqrt{lm} \in$ Spec$(X)$. Moreover, if we denote by $\lambda_i$ the eigenvalues of a graph, for a connected $k$-regular graph we have $$k=\lambda_0>\lambda_1\geq \lambda_2\geq \cdots \geq -k.$$ Thus, $k$ is the largest absolute value of an eigenvalue of $X$. We denote by $\lambda(X)$ the next largest absolute value of an eigenvalue. 
If $X$ is bipartite, the spectrum is symmetric and  $-k$ is an eigenvalue. Let $X$ be a 
finite connected bigraph with bidegree $(l,m)$, $l \geq m$. Suppose $X$ has $n_1$ vertices of degree $l$ and $n_2$ vertices of degree $m$. We must have   $n_2 \geq n_1$. Then, 
Spec$(X)$ is the multiset $$\{\pm \lambda_0, \pm \lambda_1, \ldots, \pm \lambda_{n_1},\underbrace{0, \ldots , 0}_{n_2-n_1}\},$$ where $
\lambda_0=\sqrt{lm}>\lambda_1 \geq \cdots \geq \lambda_{n_1}\geq
0$.
Then, with the above notation,  $\lambda(X)=\lambda_1$.

If $W$ is a subset of $V$, the \textit{boundary} of $W$, denoted by $\partial W$, is the set of vertices outside of $W$ which are connected by an edge to a vertex in $W$, \textit{i.e.}, $$\partial(W)=\{v\in V\setminus W \mid \{v,w\}\in E, \mbox{ for some } w \in W\}.$$ The \textit{expansion coefficient}  of a graph $X=(V, E)$ is defined as $$c=\inf\left\{\left. \frac{|\partial W|}{\min\{|W|,|V\setminus W|\}} \ \right| \ W \subseteq V: 0 <|W|<\infty \right\}.$$ Note that, if $|V|=n$ is finite,  then $$c=\min\left\{\left. \frac{|\partial W|}{|W|} \ \right| \ W \subseteq V: 0 <|W|\leq\frac{n}{2} \right\}.$$

A graph $X=(V,E)$ is called an $(n,k,c)$-\textit{expander} if $X$ is a $k$-regular graph on $n$ vertices with expansion coefficient $c$. The expansion coefficient $c$ of a regular graph is related to $\lambda(X)$, the second largest absolute value of an eigenvalue \cite[Proposition 1.2]{lps:1986} by $$2c=1-\frac{\lambda(X)}{k}.$$ Good expanders have large  expansion coefficient. Thus, good expanders have small $\lambda(X)$ (or large spectral gap, $k-\lambda(X)$). Alon and Boppana \cite{MR875835,lps:1988} showed that asymptotically $\lambda(X)$ cannot be arbitrarily small. They proved that, if $X_{n,k}$ is a $k$-regular graph with $n$ vertices, then $$\liminf_{n \to \infty} \lambda(X_{n,k})\geq 2 \sqrt{k-1}.$$ Lubotzky, Phillips and Sarnak \cite{lps:1986} defined a Ramanujan graph to be a graph that beats the Alon-Boppana bound. 

\begin{defn} A $k$-regular graph $X$ is called a Ramanujan graph if $\lambda(X) \leq 2\sqrt{k-1}$.
\end{defn}
Feng and Li \cite{feng:1996} proved the analog to the Alon-Boppana bound for  biregular bipartite graphs. They showed that,  if $X_{n,l,m}$ is a $(l,m)$-biregular graph with $n$ vertices, then $$\liminf_{n \to \infty} \lambda(X_{n,l,m})\geq  \sqrt{l-1}+\sqrt{m-1}.$$ Then, Sol\'e \cite{sole:1999}  defines Ramanujan bigraphs as the graphs that beat the Feng-Li bound. 

\begin{defn} A finite, connected, bigraph $X$ of bidegree $(l,m)$ is a Ramanujan bigraph if $$|\sqrt{l-1}-\sqrt{m-1}|\leq \lambda(X)\leq \sqrt{l-1}+\sqrt{m-1}.$$
\end{defn}
Sol\'e's definition is equivalent to the following  definition given by  in Hashimoto \cite{hashimoto:1989}.
 
 \begin{defn} A finite, connected, bigraph of bidegree $(q_1+1,q_2+1)$ is a Ramanujan bigraph if $$|(\lambda(X))^2-q_1-q_2|\leq 2\sqrt{q_1q_2}.$$
 \end{defn} 
 
Our goal is to construct an infinite family of Ramanujan bigraphs of the same bidegree and with the number of vertices growing without bound. 
  In general, it is difficult to check that a large regular or biregular graph is Ramanujan. In this article, the graphs are quotients of the Bruhat-Tits building attached to an inner form of the special unitary group in three variables. We then employ a result of \cite{ballantine:2011}, which uses the structure of the group, to estimate the spectrum of the building quotient in order to conclude that the graphs constructed are Ramanujan.

 \subsection{Unitary groups in three variables}
We denote by $F$ a local or global field of characteristic zero. For a detailed discussion on unitary groups, we refer the reader 
to \cite{rog:1990}. Let $E/F$ be a quadratic extension and $\phi: E^3
\times E^3 \to E$ be a Hermitian form. Then the \emph{special unitary group}
with respect to $\phi$ is an algebraic group over $F$ whose functor of
points is given by 
$$SU(\phi, R)=\{g \in SL_3(E\otimes_F R) \mid \phi(gx, gy) = \phi(x,
y)\  \forall x, y \in E^3 \otimes_F R\}$$ 
for any $F$-algebra $R$.
We use $SU_3$ to denote the standard
special unitary group corresponding to the Hermitian form given by the
identity matrix; that is, 
 $$SU_3(R)=\{g \in SL_3(E\otimes_F R) \mid \,^t\bar g g=\Id_3\}$$ where $\bar g$ is conjugation  with respect to the extension $E/F$.

 Let $D$ be a  central simple algebra  of degree three over $E$  and $\alpha$ be an involution of the second kind, \textit{i.e.}, 
 an anti-automorphism of $D$  that acts on $E$ by conjugation with respect to $E/F$.  By Wedderburn's theorem 
 \cite[Theorem 19.2]{knus:1998}, $D$ is a cyclic algebra over $E$. 
Let $N_D$ denote  the reduced norm of $D$. Then, $(D, \alpha)$ defines a special unitary group $\G$ by  
$$\G(R)=\{d\in (D\otimes_FR)^\times \mid \alpha(d)d=1, \ N_{D\otimes_F R}(d)=1\}.$$  Moreover, all special unitary groups are obtained in this way 
from $(D, \alpha)$ \cite[section 1.9]{rog:1990}. 
  
\subsection{Buildings} Let $F$ be a non-archimedean local field and let  $E/F$ be an \textit{unramified} separable quadratic 
extension. Let $G=SU_3$ be defined as above. Let $\mathcal O=\mathcal{O}_E$ be the ring of integers of $E$ and $\mathfrak{p}$ be the unique 
maximal  ideal in $\mathcal{O}$. Let $k=\mathcal{O}/\mathfrak{p}$ be the residue field.  
We denote by $B$ the Borel subgroup of upper-triangular matrices and by $B(k)$ the $k$-points of $B$. 
We denote by $I$ the preimage of $B(k)$ under the reduction mod $\mathfrak{p}$ map 
$G(\mathcal{O}) \to G(k)$. The group $I$ is an Iwahori subgroup. 
 Then the Weyl group $W$ of $G$ is the infinite dihedral group. 
 Let $s_1$ and $s_2$ be the reflections generating $W$.  For $i=1,2$, let $U_i=I \cup Is_iI$. 
 These subgroups are the representatives of the $G$-conjugacy classes of maximal compact subgroups of 
 $G$ \cite{has-hori:1989}. Moreover, $I=U_1\cap U_2$. 
 
 The Bruhat-Tits building associated with $G$ is a one dimensional simplicial complex defined as follows.  
 The set of $0$-dimensional simplices consists of one vertex for each maximal compact subgroup of $G$. 
 If $K_1$ and $K_2$ are two maximal compact subgroups of $G$, we place an edge between the vertices corresponding to 
 $K_1$ and $K_2$ if and only if $K_1 \cap K_2$ is conjugate to $I$ in $G$. The edges form the set of $1$-dimensional simplices of the 
 building. Since they are the faces of the largest dimension, they are the chambers of the building. 
 The group $G$ acts simplicially on the building in a natural way. The building associated with $SU_3$ is a $(q^3+1, q+1)$ tree, where $q$ is the cardinality of the residue field $k$. For more details on buildings we refer the reader to \cite{tits:1979} and \cite{garrett:1997}.

 \subsection{Ramanujan bigraphs from buildings}
 
  Let $G$ be the group $SU_3$ over $\Q_p$
(or a finite extension of $\Q_p$). Let $\tilde X$ be the Bruhat-Tits tree of $G$. 
Let $E$ be an imaginary quadratic extension of $\Q$ and let $D$ be a central simple algebra of degree $3$ over $E$ 
and $\alpha$ an involution of the second kind on $D$. Let $\G$ be the
special unitary group over $\Q$ determined by $(D,\alpha)$.  
We have the following theorem of Ballantine and Ciubotaru \cite{ballantine:2011} that motivates our work. 
\begin{thm} \cite[Theorem 1.2]{ballantine:2011} 
Let $\Gamma$ be a discrete, co-compact subgroup of $G$ which acts on $G$ without fixed points.
Assume that  $D \neq M_3(E)$, $\G(\mathbb Q_p)=G$ and $\G(\R)$ is compact. Then the quotient tree $X=\tilde X/\Gamma$ is a 
Ramanujan bigraph.
\end{thm}
In the rest of this article we give a description of an  algebra $D$ together with an involution $\alpha$ fulfilling the 
assumptions of the above theorem, as well as an infinite collection of discrete, co-compact subgroups of $G$ which act on 
$G$ without fixed points. 
\medskip

\section{Choosing the algebra and the involution}

The goal of this section is to
determine {\it explicitly} a global division algebra $D$ which is central simple of degree three over its center $E$ and is 
equipped with an involution $\alpha$ of the second kind with fixed field $F$ such that the related special unitary group $\G$,
\[
 \G(R)=\{ d\in (D\otimes_F R)^\times \mid \alpha(d)d=1\textrm{ and }
 N_{D\otimes_F R} (d)=1\},
\]
yields compactness at infinity. Such an algebra exists by the Hasse principle (see for example \cite[p. 657]{harris-labesse:2004}),
which actually is much stronger: For any set of local data, there is a global one localizing to it.
We note that in \cite{ballantine:2011} the authors refer to \cite{cht:2008} for the existence of the global group 
(and thus the algebra defining it). The example of central simple algebra with involution given in \cite{ballantine:2011}  
does not necessarily lead to Ramanujan bigraphs. It is not a division algebra and  the resulting unitary group has non-tempered 
representations occurring as local components of automorphic representations. Therefore, Rogawski's 
Theorem \cite[Theorem 14.6.3]{rog:1990} does not apply.

\subsection{Cyclic central simple  algebras of degree three}
Let $E$ be a number field.
Let $L$ be a cyclic algebra of degree three over $E$, and let $\rho$ be a generator of  its automorphism group which is isomorphic to the cyclic group $C_3$. 
Then, define a cyclic central simple algebra $D$ 
of degree three over $E$ by
\[
 D=L\oplus Lz\oplus Lz^2,
\]
where $z$ is a generic element satisfying $z^3=a\in E^\times$
subject to the relation $$zl=\rho(l)z  \textrm{ for any } l\in L.$$
By a theorem of Wedderburn \cite[Theorem 19.2]{knus:1998}, any central simple algebra of degree three is cyclic.
From now on  we will assume $D$ is in the form given above.
As $D$ is a vector space over $L$ with basis $\{1,z,z^2\}$, we write the multiplication by $d\in D$ from the right in terms of
matrices   to obtain an embedding $D\hookrightarrow M_3(L)$, 

\[
 d=l_0+l_1z+l_2z^2 \mapsto A(l_0,l_1,l_2):=\begin{pmatrix}
                                            l_0&l_1&l_2\\
                                            a\rho(l_2)&\rho(l_0)&\rho(l_1)\\
                                            a\rho^2(l_1)&a\rho^2(l_2)&\rho^2(l_0)
                                           \end{pmatrix},
\]

\noindent for $l_0,l_1,l_2\in L$.
Let $N_{L/E}$ denote the norm of $L/E$, and $Tr_{L/E}$ denote the trace of $L/E$. Then  for the reduced norm of $D$ we have
 $$N_D(d)=\det A(l_0,l_1,l_2)=N_{L/E}(l_0)+aN_{L/E}(l_1)+a^2N_{L/E}(l_2)-a \, Tr_{L/E}(l_0\rho(l_1)\rho^2(l_2)).$$
In order for $D$ to be  a division algebra, we  have to assume that $L/E$ is a field extension. Since $L$ is a cyclic $C_3$-algebra over $E$, it follows that $L/E$ is $C_3$-Galois.
Additionally,  $D$ is a division algebra if and only if neither $a$ nor $a^2$  belongs to the norm group $N_{L/E}$ of $L/E$ \cite[p.279]{pierce:1982}. 

\subsection{Involutions of the second kind}
Let $E/F$ be a quadratic extension of number fields, and let $\langle\tau\rangle\cong C_2$ be its Galois group.
In order to equip a division algebra $D$ over $E$ with an involution $\alpha$ of the second kind with fixed field $F$, 
we need to extend the nontrivial automorphism $\tau$ of $E$ to $D$.

We start by extending $\tau$ to $L$, 
 $\tau:L\to D$.  For this we have two possibilities. Either,  
the image $L':=\tau(L)$  equals $L$ or it does not.
If  $L'$ does not equal $L$, then $\tau$ gives rise to an isomorphism of $L$ to  $L'$ inside some field extension containing both.
However, $L$ and $L'$ are not isomorphic as extensions of $E$, otherwise $D$ would not be a division algebra.
So $L/F$ is not Galois.  Notice that in this case $L$ along with $L'$ generate $D$.
In contrast, if we extend $\tau: L\to L$, \textit{i.e.}, $\tau(L)=L$, then $\langle\tau,\rho\rangle$ is an automorphism group of $L/\Q$ of order at least six. That is, 
the degree six extension $L/F$ is Galois with Galois group $\langle\tau,\rho\rangle$, which is isomorphic to the cyclic group $C_6$ 
or the symmetric group $S_3$.

\subsection{Compactness at infinity}
We now assume $F$ is totally real. For simplicity, let $F=\Q$.

In order for the unitary group defined by $(D, \alpha)$ to be compact at infinity, we need $E/\Q$ to be imaginary quadratic. To see this, 
 assume $E/\Q$ is real quadratic. Then, $E_\infty=E\otimes_\Q\R\cong \R\oplus \R$ would split. Therefore,  $L_\infty$ would split as well 
and we would be able to  find an isomorphism $D_\infty\cong M_3(\R)\oplus M_3(\R)$, where the involution  is given 
by (see \cite[p.83]{platonovRapinchuk:1994}) $$(x,y)\mapsto (^ty,\,^tx),$$  and the reduced norm is  
given by $$N_D(x,y)=\det(x)\det(y).$$  

Thus, 
\[
\G_\infty:= \G(\mathbb R) = \{(x,y)\in D_\infty\mid (\,^tyx,\,^txy)=(\Id_3,\Id_3) \textrm{ and } \det(x)\det(y)=1\}\cong GL_3(\R)
\]
 is not compact.

Next we remark that in the case when $L/\Q$ is Galois, the Galois group is necessarily  $C_6$.
To see this, assume $L/E$ is a $C_3=\langle\rho\rangle$-Galois extension such that $L/\Q$ is $S_3$-Galois. 
At infinity, we have $$E_\infty=E\otimes_\Q\R\cong\C$$ and $\tau$ acts by complex conjugation. Therefore, 
$$L_\infty=L\otimes_\Q\R\cong L\otimes_E\C\cong \C\oplus\C\oplus\C,$$ with the isomorphism given by
$$l\otimes s\mapsto (\rho^0(l)s,\rho^1(l)s,\rho^2(l)s) \textrm{ for } l\in L \textrm{ and } s\in E_\infty.$$
Notice that $[L:E]=3$, so there is always a real primitive element, and thus there is  an $E$-basis for $L$ which is 
$\tau$-invariant. 
Here multiplication in $L_{\infty}$ is defined coordinate wise. The $S_3$-action is given by
$$\rho(l\otimes s)=\rho(l)\otimes s\mapsto (\rho^1(l)s,\rho^2(l)s,\rho^0(l)s)$$ and
$$\tau(l\otimes s)\mapsto(\tau(l)\tau(s),\rho^2\tau(l)\tau(s),\rho\tau(l)\tau(s)).$$ Thus, for any 
$(t_0,t_1,t_2)\in L_\infty$ we have
\begin{align*}
 \rho(t_0,t_1,t_2) &= (t_1,t_2,t_0),\\
 \tau(t_0,t_1,t_2))&=(\bar t_0,\bar t_2,\bar t_1),
\end{align*}
with the usual complex conjugation. Without specifying the  algebra $(D,\alpha)$ containing $L$ any further, we read off
that $D_\infty$ is isomorphic to the matrix algebra $M_3(\C)$ with $L_\infty$ embedded diagonally. This leads to the following result. 
\begin{prop}\label{prop_noncompact_at_infinity}
Let $E$,$L$, and $(D,\alpha)$ be as above, and assume $L/\Q$ is $S_3$-Galois. Then the split torus
 \[T_\infty = \{ (\bar tt^{-1},t,\bar t^{-1})\mid t\in\C^\times\}\subset L_\infty\]
 is contained in $\G_\infty$. In particular, $\G_\infty$ is non-compact.
\end{prop}
\begin{proof}[Proof of Proposition~\ref{prop_noncompact_at_infinity}]
 We check the definition of $\G$ for elements of $T_\infty$. We have \[ N_D((\bar tt^{-1},t,\bar t^{-1}))=N_{L_\infty/E_\infty}((\bar tt^{-1},t,\bar t^{-1}))
 =\bar tt^{-1}\cdot t\cdot \bar t^{-1}=1,\]
 as well as
 \[\alpha((\bar tt^{-1},t,\bar t^{-1}))\cdot (\bar tt^{-1},t,\bar t^{-1})=
  \tau((\bar tt^{-1},t,\bar t^{-1}))\cdot(\bar tt^{-1},t,\bar t^{-1})=
  (t\bar t^{-1},t^{-1},\bar t)\cdot(\bar tt^{-1},t,\bar t^{-1})=1. \]
Therefore,  $T_\infty$ defines a non-compact torus of $\G_\infty$.  
\end{proof}
In the case when $L/\Q$ is Galois, there is an obvious (but not unique) choice of an involution of the second kind. As $\tau$  
extends to an automorphism of $L$, it is defined on any coefficient of $A(l_0,l_1,l_2)$. Thus,  the map
\[
 \alpha(A(l_0,l_1,l_2)):=\,^t\tau(A(l_0,l_1,l_2))
 =\begin{pmatrix}
   \tau(l_0)&\tau(a\rho(l_2))&\tau(a\rho^2(l_1))\\
   \tau(l_1)&\tau\rho(l_0)&\tau(a\rho^2(l_2))\\
   \tau(l_2)&\tau\rho(l_1)&\tau\rho^2(l_0)\\
  \end{pmatrix}
\]
clearly satisfies the conditions $$\alpha^2=\mathop{id}, $$ $$ \alpha(A\cdot B)=\alpha(B)\cdot \alpha(A),$$ $$ \alpha\mid_E=\tau.$$
In order for $\alpha$ to be an involution on $D$ of the second kind, we must have that the image $\alpha(D)$ is contained in $D$.
Defining 
\[
\tilde l_0=\tau(l_0),\quad \tilde l_1=\tau(a)\tau\rho(l_2),\quad \tilde l_2=\tau(a)\tau\rho^2(l_1),
\]
this condition is equivalent to
\[
 \alpha(A(l_0,l_1,l_2))=A(\tilde l_0,\tilde l_1,\tilde l_2)).
\]
That is, 
\[
 \begin{pmatrix}
   \tau(l_0)&\tau(a\rho(l_2))&\tau(a\rho^2(l_1))\\
   \tau(l_1)&\tau\rho(l_0)&\tau(a\rho^2(l_2))\\
   \tau(l_2)&\tau\rho(l_1)&\tau\rho^2(l_0)\\
  \end{pmatrix}
  =
 \begin{pmatrix}
 \tilde l_0&\tilde l_1&\tilde l_2\\
 a\rho(\tilde l_2)&\rho(\tilde l_0)&\rho(\tilde l_1)\\
 a\rho^2(\tilde l_1)&a\rho^2(\tilde l_2)&\rho^2(\tilde l_0)
 \end{pmatrix} .
\]
This evidently reduces to the following conditions $$\tau\rho=\rho\tau \textrm{ on }L$$ and $$a\tau(a)=1.$$ 
We summarize the above discussion in the following theorem. 
\begin{thm}\label{prop_one_choice_of_alpha_for_C6}
Assume the extension $L/\Q$ is Galois, and that $\alpha$ is defined by
\[
\alpha(A(l_0,l_1,l_2))=\,^t\tau(A(l_0,l_1,l_2)). 
\]
Then $(D,\alpha)$ is a division algebra which is central simple over $E$ with involution $\alpha$ of the second kind if and only if the 
following conditions are satisfied:
\begin{itemize}
 \item[(i)]
 $a\in E^\times$, and $a,a^2\notin N_{L/E}$
  \item[(ii)]
 $N_{E/\Q}(a)=a\tau(a)=1$
 \item[(iii)]
 $\tau\rho=\rho\tau$ on $L$, i.e. $L/\Q$ is $C_6$-Galois.
\end{itemize}
Moreover, if these conditions are satisfied,  the group $\G_\infty$ is compact. 
\end{thm}
\begin{proof}[Proof of Theorem~\ref{prop_one_choice_of_alpha_for_C6}] The first part of the theorem is proved above. 
What is left to show is compactness at infinity. 
The realization of $L$ inside the diagonal subgroup of $M_3(L)$ is chosen such that it is compatible with the isomorphism
$$L_\infty=L\otimes_\Q\R\cong L\otimes_E\C\cong \C^3$$ induced by
the three embeddings of $L$ into $\C$.
Indeed, $D_\infty\cong M_3(\C)$ with involution $\alpha:M_3(\C)\to M_3(\C)$ given by $\alpha(A)=\,^t\overline{A}$.
Therefore, 
\[
 \G_\infty\cong\{A\in M_3(\C)\mid \,^t\overline{A}\cdot A=\Id_3,\, \det A=1\} =SU_3(\R),
\]
is induced by the standard hermitian form of signature $(3,0)$. Thus, $\G_\infty$ is compact.
\end{proof}
Notice that it is non-trivial to satisfy condition (iii) of Theorem~\ref{prop_one_choice_of_alpha_for_C6}, as a
quadratic field $E/\Q$ does not necessarily allow an extension $L$ of degree three which is $C_6$-Galois over $\Q$.
However,  there are situations which allow for  the conditions of Theorem~\ref{prop_one_choice_of_alpha_for_C6} to be satisfied. Below we provide such an example.

\begin{example}\label{example_galois_case}
{\it An example in the Galois-case.}
 Let $E=\Q(\sqrt{-3})$. Therefore, $E$ contains a primitive third root  of unity, $\zeta_3$, and Kummer theory applies. 
 That is,  any cyclic
 $C_3$-extension $L/E$ can be obtained by adjoining a third root, $L=E(\sqrt[3]{b})$, where $b\in E^\times\backslash (E^\times)^3$.
 In particular, choose $b=\zeta_3$. Then, $\sqrt[3]{\zeta_3}=\zeta_9$ is a primitive 9th root of unity. 
 Then, $L=E(\zeta_9)=\Q(\zeta_9)$ is a cyclotomic field, which is tautologically cyclic over $\Q$.
 Its relative  Galois group is $\mathop{Gal}(L/E)=\langle\rho\rangle$, where
 $\rho(\zeta_9)=\zeta_3\zeta_9$.
 Extending $\tau$ (complex conjugation) from $E$ to $L$ means $\tau(\zeta_9)=\zeta_9^8$.
 Thus,  $$\rho\tau(\zeta_9)=\rho(\zeta_9)^8=\zeta_3^8\zeta_9^8=\bar\zeta_3\zeta_9^8=\tau\rho(\zeta_9).$$
  Now choose an element $a\in E^\times$ such that $a,a^2\notin N_{L/E}$ and $N_{E/\Q}(a)=1$.
 One can take for example 
 \[
  a=\frac{2+\sqrt{-3}}{2-\sqrt{-3}}.
 \]
Then, trivially, $N_{E/\Q}(a)=1$, and we  verified using Magma that $a,a^2\notin N_{L/E}$. 
\end{example}

\begin{example}{\it An example in the non-Galois case.}
Again, choose $E=\Q(\sqrt{-3})$. But this time, choose a cyclic degree three extension $L=E(\theta)$, $\theta^3=b$, where $b\in E^\times\backslash (E^\times)^3$ is chosen such that $L/\Q$ is not Galois.  For example, one could choose $b=2\zeta_3$. The automorphism $\rho$ of $L/E$ is given by $\rho(\theta)=\zeta_3\theta$, and the minimal polynomial is given by $X^3-b$.
Let $\theta'$ be a root of $X^3-\tau(b)$, and let $L'=E(\theta')$.
Then (within any field extension containing both) $L$ and $L'$ are non-equal, but there is an isomorphism
$\alpha:L\to L'$ extending $\tau$  given by $\tau(\theta)=\theta'$. For the cyclic algebra  $(D,\alpha)$ with involution, choose $L$ as above and $a=\tau(b)$, i.e. $z$ may be identified with $\theta'$. Then the above constraint
\[
\alpha(\theta)=z
\]
determines an involution on $D$ of the second kind, as $\alpha(z)=\alpha^2(\theta)=\theta$.
For convenience, let $d=l_0+l_1z+l_2z^2$, $l_j\in L$, be an arbitrary element of $D$, then
\[
\alpha(d)=\alpha(l_0)+\theta\alpha(l_1)+\theta^2\alpha(l_2),
\] 
and one easily checks $\alpha^2(d)=d$. Using the identification of $D$ with a subring of $M_3(L)$ as before, we  write down this involution for matrices:
\[z=\begin{pmatrix}0&1&0\\0&0&1\\a&0&0\end{pmatrix}\mapsto 
\alpha(z)= \begin{pmatrix}\theta&0&0\\0&\rho(\theta)&0\\0&0&\rho^2(\theta)\end{pmatrix}.
\] 
So for an element $e_0+e_1z+e_2z^2\in L'=E(z)\subset D$, i.e. $e_j\in E$,
\[
\alpha(A(e_0,e_1,e_2))
=A(\tau(e_0)+\tau(e_1)\theta+\tau(e_2)\theta^2,0,0).
\]
As $\alpha^2=\mathop{id}$, we read off the image of $l=e_0+e_1\theta+e_2\theta^2\in L$ under $\alpha$ in matrix form:
\[
\alpha(A(l,0,0))=
\begin{pmatrix}
\tau(e_0)&\tau(e_1)&\tau(e_2)\\
a\tau(e_2)&\tau(e_0)&\tau(e_1)\\
a\tau(e_1)&a\tau(e_2)&\tau(e_0)\end{pmatrix}.
\]
Thus, finally
\[
\alpha(A(l_0,l_1,l_2))=\alpha(A(l_0,0,0))+\theta\alpha(A(l_1,0,0))+\theta^2\alpha(A(l_2,0,0)),
\]
that is, for $l_j=e_{j0}+e_{j1}\theta+e_{j2}\theta^2\in L$ with $e_{jk}\in E$, we find
\[
\alpha(A(l_0,l_1,l_2))=A(\tilde l_0,\tilde l_1,\tilde l_2),
\]
where $\tilde l_j=\tau(e_{0j})+\tau(e_{1j})\theta+\tau(e_{2j})\theta^2$.
\end{example}

\section{Choosing the family of subgroups}

Let $\G$ be the global special unitary group constructed from the division algebra and the involution of the second kind given in 
Example~\ref{example_galois_case} of the previous section. Let $p$ be a place where $\G_p:=\G(\Q_p)$ is isomorphic to $SU_3(\Q_p)$. 
In this section, we will give an explicit infinite family of discrete co-compact subgroups of $\G$ which act without 
fixed points on the Bruhat-Tits tree of $\G_p$. Before we proceed, we need to describe the place $p$ explicitly. 
From \cite[14.2]{rog:1990} we have that $\G_p$ is isomorphic to $SU_3(\Q_p)$ if and only if $p$ is  inert in $E$. 
In fact, we can see this directly as shown below. 
If $p$ does not remain prime (i.e. is not inert), then there are two cases.
Either (i) $p$ ramifies in $E$ (i.e. $(p)=\mathfrak p^2$, $\mathfrak p=\overline{\mathfrak p}$) or  (ii) $p$  splits into 
two non-equal prime ideals in $E$ (i.e. $(p)=\mathfrak p\overline{\mathfrak p}$ with
$\mathfrak p\not=\overline{\mathfrak p}$). 

(i) The only prime ramified in $E$ is  $(p)=(3)=\mathfrak p^2$, where $\mathfrak p=(\sqrt{-3})=(-\sqrt{-3})=\overline{\mathfrak p}$. In this case,  $E_{\mathfrak p}/\Q_p$ is a ramified 
field extension. The involution $\alpha$ is then trivial on the localization $E_{\mathfrak p}$, 
as $\alpha(\mathfrak p)=\overline{\mathfrak p}=\mathfrak p$. The group $\G_p$ will lead to a $(p+1)$-regular tree, the Bruhat-Tits building on $SL_2(\Q_p)$. This case has been treated in \cite{lps:1988}. 

(ii) There are many primes $p$ which are split in $E=\Q(\sqrt{-3})$.  The minimal polynomial is $X^2+3$. 
This is reducible modulo $p$ if and only if $p$ is split in $E$. Equivalently, the minimal polynomial is reducible if and only if  $-3$ is a square mod $p$.
The two localizations $E_\mathfrak p$ and $E_{\overline{\mathfrak p}}$ here are both equal to $\Q_p$. Therefore,  the field extension
$E/\Q$ localizes as a split algebra $E_p=E_\mathfrak p\oplus E_{\overline{\mathfrak p}}=\Q_p\oplus \Q_p$. The involution 
$\alpha$ is conjugation on $E$, so $\alpha(\mathfrak p)=\overline{\mathfrak p}$. That is, $\alpha$ exchanges the two summands
of $E_p$. Then, $D_p=D_\mathfrak p\oplus D_{\overline{\mathfrak p}}$, and for an element $(g_1,g_2)\in D_p$ to be in
$\G_p$ we need $$N(g_1,g_2)=(1,1)$$ and $$(1,1)=\alpha((g_1,g_2))(g_1,g_2)=(\alpha(g_2),\alpha(g_1))(g_1,g_2)=
(\alpha(\alpha(g_1)g_2),\alpha(g_1)g_2).$$ That is, $g_2=\alpha(g_1)^{-1}$ for some $g_1$ in the reduced norm one group,
$N_{D_\mathfrak p}^1$, of $D_\mathfrak p$. Thus,  $\G_p\cong N_{D_\mathfrak p}^1\cong N^1_{D_{\overline{\mathfrak p}}}$.

Finally, if $p$ is inert in $E$, then $E_p/\Q_p$ is an unramified quadratic field extension.  This is the case if and only if
$-3$ is not a square modulo $p$. Then $\G_p\cong \G(\Q_p)$. Therefore,  only these primes are ``good'' primes for us, \textit{i.e.}, 
leading to Ramanujan bigraphs. By quadratic reciprocity, for a prime $p>3$,
\[
\left( \frac{-3}{p}\right) =\begin{cases} 1, & p\equiv 1, 7 \pmod{12}
\\
-1 & p\equiv 5, 11 \pmod{12}.
\end{cases}
\]
Thus the ``good'' primes are the primes $p$ such that $p \equiv 5, 11 \pmod{12}$. 

Fix a prime $p \equiv 5, 11 \pmod{12}$ and let $q$ be a prime not equal to $p$. 
We follow the notation in  \cite[4.3]{ballantine:2011}.  Let $\Z[p^{-1}]$ be the subring of $\Q$ consisting of rational numbers 
with  powers of $p$ in the denominator.  
Notice that $\G_\infty$ and $\G_p$ are matrix
groups with coefficients in $\R$ and $\Q_p$, respectively. By abuse of notation, we denote by $\G_\infty(\Z[p^{-1}])$ 
and $\G_p(\Z[p^{-1}])$ the obvious subgroups  in $\G_\infty$ and $\G_p$, respectively. It is clear that  $\G_\infty(\Z[p^{-1}])$ and $
\G_p(\Z[p^{-1}])$ are isomorphic. Define
$\G(\Z[p^{-1}]):=\G_\infty(\Z[p^{-1}])\times \G_p(\Z[p^{-1}])$ to be their product in $\G_\infty\times \G_p$.
It follows from   
\cite{borel:1963} that $\G(\Z[p^{-1}])$ is a lattice in $\G_\infty\times \G_p$. 
For each positive integer $n$, 
we define the kernel modulo $q^n$, $$\Gamma(q^n):=\ker(\G(\Z[p^{-1}]) \to \G(\Z[p^{-1}]/q^n\Z[p^{-1}]),$$ and $$\Gamma_p(q^n):=\Gamma(q^n)\cap \G_p.$$ 
Then, as shown in \cite{ballantine:2011}, each $\Gamma_p(q^n)$ is a discrete co-compact subgroup of $\G_p$. 
It has finite index and no nontrivial elements of finite order. Thus, each subgroup $\Gamma_p(q^n)$  acts on the 
Bruhat-Tits tree of $\G_p$ without fixed points and the quotient building is a finite biregular graph of bidegree $(p^3+1,p+1)$.

\section{An infinite family of Ramanujan bigraphs}

Let $\G$ be the inner form of $SU_3$ constructed using the division algebra and involution of Example 3.3. 
Let $p$ be a prime congruent to $5$ or $11$ modulo $12$ and $q$ a prime not equal to $p$.  We denote by $\tilde{X}$  
the Bruhat-Tits tree associated with $\G_p$. For each positive integer $n$, let $\Gamma_p(q^n)$ be the subgroup of 
$\G_p$ constructed in the previous section and let $X_n$ be the quotient of $\tilde{X}$ by the action of $\Gamma_p(q^n)$. 
By \cite[Corollary 4.6]{ballantine:2011}, $X_n$ is a Ramanujan bigraph. Thus, we have constructed an infinite family of Ramanujan 
bigraphs. As $\Gamma_p(q^{n+1})\subsetneq \Gamma_p(q^n)$, the number of vertices of $X_n$ tends to infinity as $n\to \infty$. 
Moreover, for each $n$, the graph $X_{n}$ is a subgraph of $X_{n+1}$.

\providecommand{\bysame}{\leavevmode\hbox to3em{\hrulefill}\thinspace}
\providecommand{\MR}{\relax\ifhmode\unskip\space\fi MR }
\providecommand{\MRhref}[2]{%
  \href{http://www.ams.org/mathscinet-getitem?mr=#1}{#2}
}
\providecommand{\href}[2]{#2}

\end{document}